\documentclass[graybox]{svmult}

\usepackage{mathptmx}       
\usepackage{helvet}         
\usepackage{courier}        
\usepackage{type1cm}        
\usepackage{makeidx}         
\usepackage{graphicx}        
\usepackage{multicol}        
\usepackage[bottom]{footmisc}

\makeindex             

\begin{document}

\title*{Ergodicity properties of $p$ -adic $(2,1)$-rational dynamical systems with unique fixed point}
\titlerunning{Ergodicity properties of $p$ -adic $(2,1)$-rational dynamical systems}
\author{Iskandar A. Sattarov}
\institute{I.A. Sattarov \at Institute of mathematics,
81, Mirzo Ulug'bek str., 100170, Tashkent, Uzbekistan. \\ \email{sattarovi-a@yandex.ru}}
%
%
\maketitle

\abstract*{We consider a family of $(2,1)$-rational functions given on the
set of $p$-adic field $Q_p$. Each such function has a unique
fixed point. We study ergodicity properties of the dynamical
systems generated by $(2,1)$-rational functions. For each such
function we describe all possible invariant spheres. We
characterize ergodicity of each $p$-adic dynamical system with
respect to Haar measure reduced on each invariant sphere. In
particular, we found an invariant spheres on which the dynamical
system is ergodic and on all other invariant spheres the dynamical
systems are not ergodic.}

\abstract{We consider a family of $(2,1)$-rational functions given on the
set of $p$-adic field $Q_p$. Each such function has a unique
fixed point. We study ergodicity properties of the dynamical
systems generated by $(2,1)$-rational functions. For each such
function we describe all possible invariant spheres. We
characterize ergodicity of each $p$-adic dynamical system with
respect to Haar measure reduced on each invariant sphere. In
particular, we found an invariant spheres on which the dynamical
system is ergodic and on all other invariant spheres the dynamical
systems are not ergodic.}

\section{Introduction}

In this paper we will state some results concerning discrete
dynamical systems defined over the $p$-adic field $Q_p$.
In \cite{AKTS}  the behavior  of a $p$-adic dynamical
system $f(x)=x^n$ in the fields of $p$-adic numbers $Q_p$ and
complex $p$-adic numbers $C_p$ was investigated. Some ergodic
properties of that dynamical system have been considered in
\cite{GKL}.

In \cite{MR} the behavior of the trajectory of a rational $p$-adic
dynamical system in complex $p$-adic filed $C_p$ is studied. It is proved that such
kind of dynamical system is not ergodic on a unit sphere with
respect to the Haar measure.

In \cite{ARS}, \cite{RS}, \cite{RS2} and \cite{S}  the trajectories of some rational
$p$-adic dynamical systems in a complex $p$-adic field $C_p$ are studied.

In this paper for a class of $(2,1)$-rational functions we study
ergodicity properties of the dynamical systems on the sphere of
$p$-adic numbers $Q_p$. For each such function we describe all
possible invariant spheres. We characterize ergodicity of each
$2$-adic dynamical system with respect to Haar measure reduced on
each invariant sphere. In particular, we found an invariant
spheres on which the dynamical system is ergodic and on all other
invariant spheres the dynamical systems are not ergodic.

\subsection{$p$-adic numbers}

Let $Q$ be the field of rational numbers. The greatest common
divisor of the positive integers $n$ and $m$ is denotes by
$(n,m)$. Every rational number $x\neq 0$ can be represented in the
form $x=p^{\gamma(x)}\frac{n}{m}$, where $\gamma(x),n\in Z$, $m$ is a
positive integer, $(p,n)=1$, $(p,m)=1$ and $p$ is a fixed prime
number.

The $p$-adic norm of $x$ is given by
$$
|x|_p=\left\{
\begin{array}{ll}
p^{-\gamma(x)}, & \ \textrm{ for $x\neq 0$},\\[2mm]
0, &\ \textrm{ for $x=0$}.\\
\end{array}
\right.
$$
It has the following properties:

1) $|x|_p\geq 0$ and $|x|_p=0$ if and only if $x=0$,

2) $|xy|_p=|x|_p|y|_p$,

3) the strong triangle inequality holds
$$
|x+y|_p\leq\max\{|x|_p,|y|_p\},
$$

3.1) if $|x|_p\neq |y|_p$ then $|x+y|_p=\max\{|x|_p,|y|_p\}$,

3.2) if $|x|_p=|y|_p$ then $|x+y|_p\leq |x|_p$.

Thus $|x|_p$ is a non-Archimedean norm.

The completion of $Q$ with  respect to the $p$-adic norm defines
the $p$-adic field which is denoted by $Q_p$.

For any $a\in Q_p$ and $r>0$ denote
$$
U_r(a)=\{x\in Q_p : |x-a|_p\leq r\},\ \ V_r(a)=\{x\in Q_p :
|x-a|_p< r\},
$$
$$
S_r(a)=\{x\in Q_p : |x-a|_p= r\}.
$$

A function $f:U_r(a)\to Q_p$ is said to be {\it analytic} if it
can be represented by
$$
f(x)=\sum_{n=0}^{\infty}f_n(x-a)^n, \ \ \ f_n\in Q_p,
$$ which converges uniformly on the ball $U_r(a)$.

\subsection{Dynamical systems in $Q_p$}

In this section we recall some known facts concerning dynamical
systems $(f,U)$ in $Q_p$, where $f: x\in U\to f(x)\in U$ is an
analytic function and $U=U_r(a)$ or $Q_p$.

Now let $f:U\to U$ be an analytic function. Denote $x_n=f^n(x_0)$,
where $x_0\in U$ and $f^n(x)=\underbrace{f\circ\dots\circ
f}_n(x)$.

Let us first recall some  the standard terminology of the theory
of dynamical systems (see for example \cite{PJS}). If $f(x_0)=x_0$
then $x_0$ is called a {\it fixed point}.
A fixed point $x_0$ is called an
{\it attractor} if there exists a neighborhood $V(x_0)$ of $x_0$
such that for all points $y\in V(x_0)$ it holds that
$\lim\limits_{n\to\infty}y_n=x_0$. If $x_0$ is an attractor then
its {\it basin of attraction} is
$$
A(x_0)=\{y\in Q_p :\ y_n\to x_0, \ n\to\infty\}.
$$

A fixed point $x_0$ is called {\it repeller} if there  exists a
neighborhood $V(x_0)$ of $x_0$ such that $|f(x)-x_0|_p>|x-x_0|_p$
for $x\in V(x_0)$, $x\neq x_0$.

A set $V$ is called
an invariant for $f$, if $f(V)\subset V$.

Let $x_0$ be a fixed point of a
function $f(x)$. The ball $V_r(x_0)$ (contained in $U$) is said to
be a {\it Siegel disk} if each sphere $S_{\rho}(x_0)$, $\rho<r$ is an
invariant sphere for $f(x)$. The union of all Siegel disks with the center at $x_0$ is
said to {\it a maximum Siegel disk} and is denoted by $SI(x_0)$.

Let $x_0$ be a fixed point of an analytic function  $f(x)$. Put
$$
\lambda=\frac{d}{dx}f(x_0).
$$

The point $x_0$ is {\it attractive} if $0\leq |\lambda|_p<1$, {\it indifferent} if
$|\lambda|_p=1$, and {\it repelling} if $|\lambda|_p>1$.

\section{Ergodicity of $(2,1)$-Rational $p$-adic dynamical systems}

A function is called an $(n,m)$-rational function if and only if
it can be written in the form $f(x)={P_n(x)\over T_m(x)}$, where
$P_n(x)$ and $T_m(x)$ are polynomial functions with degree $n$ and
$m$ respectively ($T_m(x)$ is a non zero polynomial).

In this paper we consider the ergodicity properties of the
dynamical system associated with the $(2,1)$-rational function
$f:Q_p\to Q_p$ defined by
\begin{equation}\label{3.1}
f(x)=\frac{x^2+ax+b}{x+c}, \ \ a,b,c\in Q_p,\ \  a\neq c, \  \
c^2-ac+b\ne 0
\end{equation}

where  $x\neq \hat x=-c$.

Note that $f(x)$ has the unique fixed point $x_0={b\over{c-a}}$.

For any $x\in Q_p$, $x\neq \hat x=-c$, by simple calculations we
get
\begin{equation}\label{ff}
    |f(x)-x_0|_p=|x-x_0|_p\ \cdot{|(x-x_0)+(x_0+a)|_p\over {|(x-x_0)+(x_0+c)|_p}}.
\end{equation}

Denote
$$
\mathcal P=\{x\in Q_p: \exists n\in N\cup\{0\}, f^n(x)=\hat x\},
$$
$$\alpha=|x_0+a|_p \ \ {\rm and} \ \ \beta=|x_0+c|_p.$$
Consider the following functions (see \cite{ARS}):

For $0\leq \alpha<\beta$ define the function $\varphi_{\alpha,\beta}: [0,+\infty)\to [0,+\infty)$ by
$$\varphi_{\alpha,\beta}(r)=\left\{\begin{array}{lllll}
{\alpha\over\beta}r, \ \ {\rm if} \ \ r<\alpha,\\[2mm]
\alpha^*, \ \ {\rm if} \ \ r=\alpha,\\[2mm]
{r^2\over\beta}, \ \ {\rm if} \ \ \alpha<r<\beta,\\[2mm]
\beta^*, \ \ {\rm if} \ \ r=\beta,\\[2mm]
r, \ \ \ \ {\rm if} \ \ r>\beta,
\end{array}
\right.
$$
where $\alpha^*$ and $\beta^*$ are some given numbers with $\alpha^*\leq{\alpha^2\over\beta}$, $\beta^*\geq\beta$.

For $0\leq \beta<\alpha$ define the function $\phi_{\alpha,\beta}: [0,+\infty)\to [0,+\infty)$ by
$$\phi_{\alpha,\beta}(r)=\left\{\begin{array}{lllll}
{\alpha\over\beta}r, \ \ {\rm if} \ \ r<\beta,\\[2mm]
\beta', \ \ {\rm if} \ \ r=\beta,\\[2mm]
\alpha, \ \ {\rm if} \ \ \beta<r<\alpha,\\[2mm]
\alpha', \ \ {\rm if} \ \ r=\alpha,\\[2mm]
r, \ \ \ \ {\rm if} \ \ r>\alpha,
\end{array}
\right.
$$
where $\alpha'$ and $\beta'$ some positive numbers with $\alpha'\leq \alpha$, $\beta'\geq\alpha$.

For $\alpha\geq 0$ we define the function $\psi_{\alpha}: [0,+\infty)\to [0,+\infty)$ by
$$\psi_{\alpha}(r)=\left\{\begin{array}{ll}
r, \ \ {\rm if} \ \ r\ne\alpha,\\[2mm]
\hat\alpha, \ \ {\rm if} \ \ r=\alpha,\\[2mm]
\end{array}
\right.
$$
where $\hat\alpha$ is a given number.

Using the formula (\ref{ff}) we easily get the following:

\begin{lemma}\label{l1} If $x\in S_r(x_0)$, then the following formula holds
$$|f^n(x)-x_0|_p=\left\{\begin{array}{lll}
\varphi_{\alpha,\beta}^n(r), \ \ \mbox{if} \ \ \alpha<\beta,\\[2mm]
\phi_{\alpha,\beta}^n(r), \ \ \mbox{if} \ \ \alpha>\beta,\\[2mm]
\psi_{\alpha}^n(r), \ \ \ \ \mbox{if} \ \ \alpha=\beta.
\end{array}\right.  \ \ n\geq 1.$$
\end{lemma}
Thus the $p$-adic dynamical system $f^n(x), n\geq 1, x\in Q_p, x\ne \hat x$ is
related to the real dynamical systems generated by $\varphi_{\alpha,\beta}$, $\phi_{\alpha,\beta}$ and $\psi_\alpha$.

\begin{theorem}\label{t1}\cite{ARS}
The $p$-adic dynamical system generated by $f$ has the following properties:
\begin{itemize}
\item[1.] If $\alpha<\beta$, then $A(x_0)=V_{\beta}(x_0)$ and the spheres $S_r(x_0)$ are invariant with respect to $f$ for all $r>\beta$.
\item[2.] If $\alpha=\beta$, then $SI(x_0)=V_{\beta}(x_0)$ and the spheres $S_r(x_0)$ are invariant with respect to $f$ for all $r>\beta$.
\item[3.] If $\alpha>\beta$, then the inequality $|f(x)-x_0|_p>|x-x_0|_p$ satisfies
for $x\in V_{\alpha}(x_0)$, $x\neq x_0$ and the spheres $S_r(x_0)$ are invariant with respect to $f$ for all $r>\alpha$.
\item[4.] $f(S_r(x_0))\not\subset S_r(x_0)$ for any $r\in\{\alpha, \beta\}$.
\item[5.]
\begin{itemize}
\item[5.1.] If $\alpha\leq\beta$, then $\mathcal P\subset S_{\beta}(x_0)$.
\item[5.2.] If $\alpha>\beta$, then $\mathcal P\subset U_{\alpha}(x_0)$.
\end{itemize}
\end{itemize}
\end{theorem}

We define the following sets
$$I_1=\{r: \, r>\max\{\alpha,\beta\}\} \ \ {\rm if} \ \ \alpha\neq\beta;$$
$$I_2=\{r: \, r\neq\beta\} \ \ {\rm if} \ \ \alpha=\beta;$$
and we denote $I=I_1\cup I_2$.

Using the Theorem \ref{t1} we get the following

\begin{corollary}
The sphere $S_r(x_0)$ is invariant for $f$ if
and only if $r\in I$.
\end{corollary}

In this paper we are
interested to study ergodicity properties of the dynamical system on the invariant sphere.

\begin{lemma}\label{l2} For every closed ball $U_{\rho}(s)\subset S_r(x_0), \, r\in I$ the
following equality holds $$f(U_{\rho}(s))=U_{\rho}(f(s)).$$
\end{lemma}
\begin{proof}
From inclusion $U_{\rho}(s)\subset S_r(x_0)$ we get $|s-x_0|_p=r$.

Let $x\in U_{\rho}(s)$, i.e. $|x-s|_p\leq\rho$, then
\begin{equation}{\label{ab1}}
|f(x)-f(s)|_p=|x-s|_p\cdot\frac{|(s-x_0)(x-x_0)+(x_0+c)[(x-x_0)+(s-x_0)]+(x_0+c)(x_0+a)|_p}{|[(x-x_0)+(x_0+c)][(s-x_0)+(x_0+c)]|_p}.
\end{equation}
We have $|x-x_0|_p=r$, because $x\in U_{\rho}(s)\subset S_r(x_0)$. Consequently,
$$|f(x)-f(s)|_p=|x-s|_p\cdot\frac{\max\{r^2, \, \beta r, \, \alpha\beta\}}{(\max\{r, \, \beta\})^2}.$$

If $r\in I_1$, then $\max\{r^2, \,
\beta r, \, \alpha\beta\}=r^2$ and $\max\{r, \, \beta\}=r.$
Using this equality by (\ref{ab1}) we get $|f(x)-f(s)|_p=|x-s|_p\leq\rho.$

If $\alpha=\beta$, then $r\in I_2$. Consequently,
$r<\beta$ or $r>\beta$.

If $r<\beta$, then $\max\{r^2, \,
\beta r, \, \alpha\beta\}=\beta^2$ and $\max\{r, \, \beta\}=\beta.$
Then we get $|f(x)-f(s)|_p=|x-s|_p\leq\rho.$

If $r>\beta$, then $\max\{r^2, \,
\beta r, \, \alpha\beta\}=r^2$ and $\max\{r, \, \beta\}=r.$
Consequently $|f(x)-f(s)|_p=|x-s|_p\leq\rho.$ This completes the
proof.
\end{proof}

Recall that $S_r(x_0)$ is invariant with respect to $f$ iff $r\in
I$.
\begin{lemma}{\label{ab2}}
If $x\in S_r(x_0)$, where $r\in I$, then
$$|f(x)-x|_p=\left\{\begin{array}{lll}
{{|a-c|_pr}\over{\beta}}, \ \ \mbox{if} \ \ 0<r<\beta=\alpha\\[2mm]
|a-c|_p, \ \ \mbox{if} \ \ r>\beta=\alpha\\[2mm]
\max\{\alpha,\beta\}, \ \ \mbox{if} \ \ r\in I_1
\end{array}\right.$$
\end{lemma}
\begin{proof}
Since, if $\alpha\neq\beta$, then $|a-c|_p=|(x_0+a)-(x_0+c)|_p=\max\{\alpha,\beta\}$.

Follows from the following equality
$$|f(x)-x|_p=\left|{{(a-c)(x-x_0)}\over{(x-x_0)+(x_0+c)}}\right|_p=\left\{\begin{array}{lll}
{{|a-c|_pr}\over{\beta}}, \ \ \mbox{if} \ \ 0<r<\beta=\alpha\\[2mm]
|a-c|_p, \ \ \mbox{if} \ \ r>\beta=\alpha\\[2mm]
\max\{\alpha,\beta\}, \ \ \mbox{if} \ \ r\in I_1
\end{array}\right.$$
\end{proof}

By Lemma \ref{ab2} we have that $|f(x)-x|_p$ depends on $r$, but
does not depend on $x\in S_r(x_0)$ itself, therefore we define $\rho(r)=|f(x)-x|_p$, if
$x\in S_r(x_0)$. Then the following theorem holds as Theorem 11 in \cite{RS2}.

\begin{theorem}{\label{ca}} If $s\in S_r(x_0), \, r\in I$ then
\begin{itemize}
\item[1.] For any $n\geq 1$ the following equality holds
\begin{equation}\label{en}
|f^{n+1}(s)-f^n(s)|_p=\rho(r).
\end{equation}
\item[2.] $f(U_{\rho(r)}(s))=U_{\rho(r)}(s).$
\item[3.] If for some $\theta>0$ the ball $U_{\theta}(s)\subset S_r(x_0)$ is an invariant for $f$,
then $$\theta\geq \rho(r).$$
\end{itemize}
\end{theorem}

For each $r\in I$ consider a measurable space $(S_r(x_0),\mathcal
B)$, here $\mathcal B$ is the algebra generated by closed
subsets of $S_r(x_0)$. Every element of $\mathcal B$ is a union of
some balls $U_{\rho}(s)\subset S_r(x_0)$.

A measure $\bar\mu:\mathcal B\rightarrow R$ is said to be
\emph{Haar measure} if it is defined by $\bar\mu(U_{\rho}(s))=\rho$.

Note that $S_r(x_0)=U_r(x_0)\setminus U_{r\over p}(x_0)$. So, we have
$\bar\mu(S_r(x_0))=r(1-{1\over p})$.

We consider normalized Haar measure:
$$\mu(U_{\rho}(s))={{\bar\mu(U_{\rho}(s))}\over{\bar\mu(S_r(x_0))}}={{p\rho}\over{(p-1)r}}.$$

By Lemma \ref{l2} we conclude that $f$ preserves the measure
$\mu$, i.e.
$$
\mu(f(U_{\rho}(s)))=\mu(U_{\rho}(s)).
$$

Consider the dynamical system $(X,T,\mu)$, where $T:X\rightarrow X$ is
a measure preserving transformation, and $\mu$ is a measure.
We say that the dynamical system is {\it ergodic} if for every invariant set $V$ we have $\mu(V)=0$ or $\mu(V)=1$ (see \cite{Wal}).

\subsection{Case $p\geq 3$.}

\begin{theorem} Let $p\geq 3$.
\begin{itemize}
\item[1.] If $r\in I_1$, then the dynamical system $(S_r(x_0), f, \mu)$ is
not ergodic.
\item[2.] If $r\in I_2$, then
\begin{itemize}
\item[2.1)] If $|a-c|_p<\beta$, then the dynamical system $(S_r(x_0), f, \mu)$ is not ergodic.
\item[2.2)] If $|a-c|_p=\beta$, then the dynamical system $(S_r(x_0), f, \mu)$ is not ergodic
for $r>\beta$.
\end{itemize}
Here $\mu$ is the normalized Haar measure.
\end{itemize}
\end{theorem}

\begin{proof}
If a sphere $S_r(x_0)$ is invariant for $f$, then by the part 2 of Theorem \ref{ca}, the ball
$U_{\rho(r)}(s)$ is invariant for any $s\in S_r(x_0)$. Using Lemma
\ref{ab2} we get
$$\mu(U_{\rho(r)}(s))={{p\rho(r)}\over
(p-1)r}=\left\{\begin{array}{lll}
{{p|a-c|_p}\over{(p-1)\beta}}, \ \ \mbox{if} \ \ r<\beta=\alpha\\[2mm]
{{p|a-c|_p}\over{(p-1)r}}, \ \ \mbox{if} \ \ r>\beta=\alpha\\[2mm]
{{p\max\{\alpha,\beta\}}\over{(p-1)r}}, \ \ \mbox{if} \ \ r\in I_1
\end{array}\right.$$

1. If $r\in I_1$, then $r>\max\{\alpha,\beta\}$. Since $r$ radius is a value of a $p$-adic norm,
we have $r\geq p\max\{\alpha,\beta\}$. Thus $0<\mu(U_{\rho(r)}(s))\leq{1\over
{p-1}}$. Therefore if $p\geq 3$, then the dynamical system
$(S_r(x_0), f, \mu)$ is not ergodic for all $r\in I_1$.

2. Let $r\in I_2$. So we have $|a-c|_p\leq\alpha=\beta$. If $|a-c|_p<\beta$, then we have $p|a-c|_p\leq\beta$ and $0<\mu(U_{\rho(r)}(s))\leq{1\over
{p-1}}$. Therefore if $p\geq 3$, then the dynamical system
$(S_r(x_0), f, \mu)$ is not ergodic for all $r\in I_2$.

If $|a-c|_p=\beta$, then $p|a-c|_p\leq r$ for $r>\beta$. So we have
$$0<\mu(U_{\rho(r)}(s))=\left\{\begin{array}{ll}
{p\over{p-1}}, \ \ \mbox{if} \ \ r<\beta\\[2mm]
\leq{1\over{p-1}}, \ \ \mbox{if} \ \ r>\beta
\end{array}\right.$$

Therefore if $|a-c|_p=\beta$, then the dynamical system
$(S_r(x_0), f, \mu)$ is not ergodic for all $r>\beta$. Theorem is proved.
\end{proof}

\subsection{Case $p=2$.}
Note that $Z_2=\{x\in Q_2: \, |x|_2\leq 1\}$. So we have $1+2Z_2=S_1(0)$. In the following theorem showed a criteria of ergodicity
of rational functions which reflect $S_1(0)$ sphere to itself:

\begin{theorem}\label{crit}\cite{M}
Let $f,g:1+2Z_2\rightarrow 1+2Z_2$ be polynomials whose
coefficients are $2$-adic integers. Set $f(x)=\sum_ia_ix^i$,
$g(x)=\sum_ib_ix^i$, and $$A_1=\sum_{i\, odd}a_i, \ \
A_2=\sum_{i\, even}a_i, \ \ B_1=\sum_{i\, odd}b_i, \ \
B_2=\sum_{i\, even}b_i.$$ The rational function $R={f\over g}$ is
ergodic if and only if one of the following situations occurs:
\begin{itemize}
\item[1.] $A_1=1(\mbox{mod} 4), \, A_2=2(\mbox{mod} 4), \, B_1=0(\mbox{mod} 4), \, and \, B_2=1(\mbox{mod} 4)$.
\item[2.] $A_1=3(\mbox{mod} 4), \, A_2=2(\mbox{mod} 4), \, B_1=0(\mbox{mod} 4), \, and \, B_2=3(\mbox{mod} 4)$.
\item[3.] $A_1=1(\mbox{mod} 4), \, A_2=0(\mbox{mod} 4), \, B_1=2(\mbox{mod} 4), \, and \, B_2=1(\mbox{mod} 4)$.
\item[4.] $A_1=3(\mbox{mod} 4), \, A_2=0(\mbox{mod} 4), \, B_1=2(\mbox{mod} 4), \, and \, B_2=3(\mbox{mod} 4)$.
\item[5.] One of the previous cases with $f$ and $g$ interchanged.
\end{itemize}
\end{theorem}

But, in this paper we will study ergodicity of the dynamical system $(S_r(x_0), f, \mu)$ for any $r\in I$.
For this purpose we can not use Theorem \ref{crit} directly, because sphere's radius is arbitrary and its center is not at $0$.

That's why we will do the following.

Let $r=p^l$ and a function $f: S_{p^l}(x_0)\rightarrow S_{p^l}(x_0)$ is given.
Denote $f\circ g=f(g(t))$.

Consider $x=g(t)=p^{-l}t+x_0, \, t=g^{-1}(x)=p^l(x-x_0)$ then
it is easy to see that $f\circ g: S_1(0)\rightarrow
S_{p^l}(x_0)$.
Consequently, $g^{-1}\circ f\circ g: S_1(0)\rightarrow
S_1(0)$.

Let $\mathcal B$ (resp. $\mathcal B_1$) be the algebra generated by closed
subsets of $S_{p^l}(x_0)$ (resp. $S_1(0)$), and $\mu$ (resp. $\mu_1$)
be normalized Haar measure on $\mathcal B$ (resp. $\mathcal B_1$).

\begin{theorem}{\cite{RS2}}{\label{erg1}} The dynamical system
$(S_{p^l}(x_0), \, f, \, \mu)$ is ergodic if and only if
$(S_1(0), \, g^{-1}\circ f\circ g, \, \mu_1)$ is ergodic.
\end{theorem}

\begin{remark} We note that Theorem \ref{crit} gives a criterion of
ergodicity for rational functions defined on the sphere with fixed
radius (=1). Our Theorem \ref{erg1} allows us to use Theorem
\ref{crit} for the spheres with an arbitrary radius.
\end{remark}

Now using the above mentioned results for
$f(x)={{x^2+ax+b}\over{x+c}}$, when $p=2$ and $f: S_r(x_0)\rightarrow
S_r(x_0)$ we prove the following theorem.

\begin{theorem}{\label{erg2}}
If $p=2$, then the dynamical system $(S_r(x_0), f, \mu)$ is ergodic iff
$\alpha\neq\beta$ and $r=2\max\{\alpha,\beta\}$.
\end{theorem}

\begin{proof} Let $r=2^l$, $\alpha=2^q$ and $\beta=2^m$. Since $\alpha=|x_0+a|_2$ and $\beta=|x_0+c|_2$, then we have $x_0+a\in
2^{-q}(1+2 Z_2)$ and $x_0+c\in 2^{-m}(1+2 Z_2)$.

In $f:S_{2^l}(x_0)\rightarrow S_{2^l}(x_0)$ we change $x$ by
$x=g(t)=2^{-l}t+x_0$. We note that $x\in S_{2^l}(x_0)$, then
$|x-x_0|_2=2^l|t|_2=2^l$, $|t|_2=1$ and the function
$g^{-1}(f(g(t))):S_1(0)\rightarrow S_1(0)$ has the following form
\begin{equation}{\label{k}}
g^{-1}(f(g(t)))={{t^2+2^l(x_0+a)t}\over{t+2^l(x_0+c)}}.
\end{equation}

For the numerator of (\ref{k}) we have $$|t^2|_2=|t|_2=1, \ \
|2^l(x_0+a)t|_2=2^{q-l} \ \ \mbox{and} \ \
|2^l(x_0+c)|_2=2^{m-l}.$$

If $r\in I_1$, then $l>q$ and $l>m$. Consequently,
$$t^2+2^l(x_0+a)t=:\gamma_1(t), \ \ \mbox{is such that} \ \ \gamma_1: 1+2 Z_2\rightarrow 1+2 Z_2$$
and
$$t+2^l(x_0+c)=:\gamma_2(t)    \ \ \mbox{is such that} \ \  \gamma_2: 1+2 Z_2\rightarrow 1+2 Z_2.$$
Let $r\in I_2$, i.e. $q=m$. Then $l>m$ or $l<m$. If $l>m$, then
$$\gamma_1, \, \gamma_2: 1+2 Z_2\rightarrow 1+2 Z_2.$$
If $l<m$, then we can write (\ref{k}) the following form
\begin{equation}{\label{kk}}
g^{-1}(f(g(t)))={{{{t^2}\over{2^l(x_0+a)}}+t}\over{{t\over{2^l(x_0+a)}}+{{x_0+c}\over{x_0+a}}}}.
\end{equation}
For the numerator of (\ref{kk}) we have
$$|t|_2=\left|{{x_0+c}\over{x_0+a}}\right|_2=1, \ \ \mbox{and} \ \
\left|{{t^2}\over{2^l(x_0+a)}}\right|_2=\left|{t\over{2^l(x_0+a)}}\right|_2=2^{l-m}\leq{1\over
2}.$$ Consequently,
$${{{t^2}\over{2^l(x_0+a)}}+t}=:\delta_1(t), \ \ \mbox{is such that} \ \ \delta_1: 1+2 Z_2\rightarrow 1+2 Z_2$$
and
$${{t\over{2^l(x_0+a)}}+{{x_0+c}\over{x_0+a}}}=:\delta_2(t)    \ \ \mbox{is such that} \ \  \delta_2: 1+2 Z_2\rightarrow 1+2 Z_2.$$

Hence the function (\ref{k}) satisfies all conditions of Theorem
\ref{crit}, therefore using this theorem we have $$A_1=1, \ \
A_2=2^l(x_0+c), \ \ B_1=2^l(x_0+a) \ \ \mbox{and} \ \ B_2=1.$$

Moreover,
$$A_1=1({\rm mod}\,4), \ \ A_2\in 2^{l-m}(1+2 Z_2), \ \
B_1\in 2^{l-q}(1+2 Z_2) \ \ \mbox{and} \ \ B_2=1({\rm
mod}\,4).$$

By this relations and Theorem \ref{crit} we get
$$A_2=0({\rm mod}\,4)\ \ \mbox{and} \ \ B_1=2({\rm mod}\,4), \ \
\mbox{i.e.} \ \ l-m\geq 2 \ \ \mbox{and} \ \ l-q=1$$ or
$$A_2=2({\rm mod}\,4)\ \ \mbox{and} \ \ B_1=0({\rm mod}\,4), \ \
\mbox{i.e.} \ \ l-m=1 \ \ \mbox{and} \ \ l-q\geq2.$$

Therefore we conclude that the dynamical system $(S_1(0), \,
g^{-1}\circ f\circ g, \, \mu_1)$ is ergodic iff $q>m$ and $l=q+1$
or $q<m$ and $l=m+1$, i.e. $\alpha>\beta$ and $r=2\alpha$ or
$\alpha<\beta$ and $r=2\beta$. Consequently, by Theorem
\ref{erg1}, $(S_r(x_0), f, \mu)$ is ergodic iff $\alpha\neq\beta$
and $r=2\max\{\alpha,\beta\}$. Theorem is proved.
\end{proof}

\section*{Acknowledgments}

The author expresses his deep gratitude to Professor U. A. Rozikov for setting up the problem and for the useful suggestions.

{}
\end{document}